\documentclass[12pt]{article}
\usepackage[latin1]{inputenc}

\usepackage{fullpage}
\usepackage{amsfonts}
\usepackage{amssymb}
\usepackage{amsmath}
\usepackage{color}
\usepackage[colors]{optsys}

\renewcommand{\PRIMAL}{{\mathcal U}}
\renewcommand{\DUAL}{{\mathcal V}}
\renewcommand{\UNCERTAIN}{{\mathcal{X}}}
\renewcommand{\primal}{u}
\renewcommand{\dual}{v}
\renewcommand{\uncertain}{x}
\renewcommand{\Lagrangian}{L}
\newcommand{\Rockafellian}{R} 
\newcommand{\PartialFunction}[2]{#1\np{#2,\cdot}} %{#1\np{\cdot,#2}} 

\title{Duality Between Lagrangians and Rockafellians}

\author{Michel De Lara \\ CERMICS, \'Ecole des Ponts, Marne-la-Vall\'ee, France}

\begin{document}

\maketitle

\begin{abstract}
  In his monograph \emph{Conjugate Duality and Optimization}, Rockafellar
  puts forward a ``perturbation + duality'' method to obtain a dual problem for an
  original minimization problem.
  First, one embeds the minimization problem into a family of perturbed problems (thus
  giving a so-called perturbation function); 
  the perturbation of the original function to be minimized has recently been called a
  Rockafellian.
  Second, when the perturbation variable belongs to a primal vector space paired, by a bilinear
  form, with a dual vector space, one builds a Lagrangian from a Rockafellian;
  one also obtains a so-called dual function (and a dual problem).
  The method has been extended from Fenchel duality to generalized convexity:
  when the perturbation belongs to a primal set paired, by a coupling
  function, with a dual set, one also builds a Rockafellian from a
  Lagrangian.
  Following these paths, we highlight a duality between Lagrangians and Rockafellians.
  Where the material mentioned above mostly focuses on moving from Rockafellian to
  Lagrangian, we treat them equally and display formulas that go both ways.
  We propose a definition of Lagrangian-Rockafellian couples.
  We characterize these latter as dual functions, with respect to a coupling,
  and also in terms of generalized convex functions.
    The duality between perturbation and dual functions is not as clear cut.
\end{abstract}

\section{Introduction}

In \cite{Rockafellar:1974}, Rockafellar developed in a systematic way 
the ``perturbation + duality'' method (see also
\cite{Laurent:1972,Joly-Laurent:1971,Ekeland-Temam:1999}). 
First, one embeds a minimization problem into a family of perturbed problems,
thus giving a so-called perturbation function;
  the bivariate perturbation of the original function to be minimized has recently been called a
  \emph{Rockafellian} \cite[Chapter~5.A]{Royset-Wets:2021}. 
  Second, when the perturbation variable belongs to a primal vector space paired, by a bilinear
  form, with a dual vector space,
  Rockafellar showed how one can build a \emph{Lagrangian} from a Rockafellian;
  one also obtains a so-called dual function (and a dual problem).

  A large part of the theory has been developed in the convex bifunction case
  \cite[Sect.~29]{Rockafellar:1974}, that is, when the Rockafellian is jointly
  convex in both variables, namely decision and perturbation, or at least when
the decision set is a linear space
    (see \cite[Sect.~30]{Rockafellar:1974},
    \cite[Chap.~11.H-I]{Rockafellar-Wets:1998},
    \cite[Chapter~5]{Royset-Wets:2021}).
This has much to do with obtaining strong duality results, for which convex
analysis offers powerful tools.
In the nonconvex case, it seems that the extension to generalized
convexity has begun with \cite{Balder:1977} that makes use of 
couplings and conjugacies (see also \cite{Singer:2006}, \cite[Sect.~3]{Martinez-Legaz:2005} and
references therein). 

  In this paper, we simply want to stress a duality between
  Lagrangians and Rockafellians, in the classic Fenchel bilinear pairing case 
  and also in the more general coupling case.
Our contribution is modest and formal.
    It is modest in that most of the material can be traced back to
    \cite{Rockafellar:1974} and then \cite{Balder:1977} --- and that we follow and slightly extend their paths --- 
and also as we do not focus on strong duality.
  It is formal in that we treat Lagrangians and Rockafellians in a symmetric fashion, and highlight a duality between them
  (Theorem~\ref{th:Lagrangian-Rockafellian_Lagrangian}).
  So, there is no real novelty in the paper, but for a symmetric
  examination of two-way relationships between Lagrangians and Rockafellians
when convexity is not assumed.

  The paper is organized as follows.
  In Sect.~\ref{Lagrangians_and_Rockafellians:_bilinear_pairing_case},
  we consider the classic Fenchel bilinear pairing case.
    We revisit some of the
results in \cite{Rockafellar:1974} and try to reformulate the
``perturbation + duality'' method with as little convexity as possible in the assumptions.
We recall and sketch how one can build a Lagrangian from a Rockafellian,
and we highlight a converse construction. 
In Sect.~\ref{Lagrangians_and_Rockafellians:_general_coupling_case},
we recall how one can build a Lagrangian from a Rockafellian, in the  case where 
the perturbation belongs to a primal set paired, by a coupling
function, with a dual set.
We also propose a converse construction, from Lagrangian to Rockafellian.
Finally, we propose a notion of Lagrangian-Rockafellian couple, and
we formally express duality between Lagrangians and Rockafellians.

\section{Lagrangians and Rockafellians: bilinear pairing case}
\label{Lagrangians_and_Rockafellians:_bilinear_pairing_case}

In~\S\ref{The_bilinear_pairing_case}, we provide background on the classic
Fenchel conjugacy.
In~\S\ref{From_Rockafellians_to_Lagrangians:_bilinear_pairing},
we sketch, in one table, how one can build a Lagrangian from a Rockafellian as %introduced and
developed in \cite{Rockafellar:1974}.
In~\S\ref{From_Lagrangians_to_Rockafellians:_bilinear_pairing},
we sketch, in one table, how one can build a Rockafellian from a Lagrangian.

\subsection{The bilinear pairing case}
\label{The_bilinear_pairing_case}

We consider two real vector spaces $\PRIMAL$ and $\DUAL$ paired, in the sense of
convex analysis, by a bilinear form 
\( \proscal{\cdot}{\cdot} : \PRIMAL \times \DUAL \to \RR \). 
The classic Fenchel conjugacy~$\star$ is defined, 
for any functions \( \fonctionprimal : \PRIMAL  \to \barRR \)
and \( \fonctiondual : \DUAL \to \barRR \), by\footnote{% 
  In convex analysis, one does not use the notation~\( \LFMr{} \),
  but simply~\( \LFM{} \). We use~\( \LFMr{} \)
  to be consistent with the notation~\eqref{eq:Fenchel-Moreau_reverse_conjugate}
  for general conjugacies.
 Also the \( +\bp{- \cdots} \) expression is here to stress the proximity with the
\( \LowPlus \bp{ - \cdots} \) expression in
  Equations~\eqref{eq:Fenchel-Moreau_conjugate_all}.}
\begin{subequations}
  \begin{align}
    \LFM{\fonctionprimal}\np{\dual} 
    &= 
      \sup_{\primal \in \PRIMAL} \Bp{ \proscal{\primal}{\dual} 
      + \bp{ -\fonctionprimal\np{\primal} } } 
      \eqsepv \forall \dual \in \DUAL
      \eqfinv
      \label{eq:Fenchel_conjugate}
    \\
    \LFMr{\fonctiondual}\np{\primal} 
    &= 
      \sup_{ \dual \in \DUAL } \Bp{ \proscal{\primal}{\dual} 
      + \bp{ -\fonctiondual\np{\dual} } } 
      \eqsepv \forall \primal \in \PRIMAL
      \eqfinv
      \label{eq:Fenchel_conjugate_reverse}
    \\
    \LFMbi{\fonctionprimal}\np{\primal} 
    &= 
      \sup_{\dual \in \DUAL} \Bp{ \proscal{\primal}{\dual} 
      + \bp{ -\LFM{\fonctionprimal}\np{\dual} } } 
      \eqsepv \forall \primal \in \PRIMAL
      \eqfinp
      \label{eq:Fenchel_biconjugate}
    \\
      \LFMrbi{\fonctiondual}\np{\dual} 
    &=
      \sup_{\primal \in \PRIMAL} \Bp{ \proscal{\primal}{\dual} 
      + \bp{ -\LFMr{\fonctiondual}\np{\primal} } } 
      \eqsepv \forall \dual \in \DUAL
      \eqfinv
      \label{eq:Fenchel_biconjugate_reverse}
 \end{align}
\end{subequations}
A function \( \fonctionprimal : \PRIMAL  \to \barRR \)
is said to be
\( \Fenchelcoupling \)-convex if
\( \LFMbi{\fonctionprimal}=\fonctionprimal \) and a function
\( \fonctiondual : \DUAL \to \barRR \) is said to be
\( \Fenchelcoupling' \)-convex if
\( \LFMrbi{\fonctiondual}=\fonctiondual \).
In both cases, this is equivalent to the function being
either one of the two constant functions~$-\infty$ and~$+\infty$
or proper convex lsc (lower semi continuous) \cite[Corollary~12.2.1]{Rockafellar:1970}.
%
% We define the subdifferential\footnote{% 
%   In convex analysis, one does not use the notation
%   \( \Midsubdifferential{\Fenchelcoupling}{\fonctionprimal}\np{\primal} \) 
%   but simply \( \subdifferential{}{\fonctionprimal}\np{\primal} \). We do it 
%   to be consistent with the notation of subdifferential
%   for general couplings.}
% \( \Midsubdifferential{\Fenchelcoupling}{\fonctionprimal}\np{\primal} \)
% \( =\subdifferential{}{\fonctionprimal}\np{\primal} \)
% of a function \( \fonctionprimal : \PRIMAL  \to \barRR \)
% at~\( \primal \in \PRIMAL \) by 
% \( \dual \in \subdifferential{}{\fonctionprimal}\np{\primal}
%            \iff \fonctionprimal\np{\primal} \UppPlus \LFM{\fonctionprimal}\np{\dual} =
%            \FenchelCoupling{\primal}{\dual} \)

\subsection{From Rockafellians to Lagrangians: bilinear pairing}
\label{From_Rockafellians_to_Lagrangians:_bilinear_pairing}

We sketch below the ``perturbation + duality'' method developed in
\cite{Rockafellar:1974}.
More precisely, we follow \cite[Sect.~4]{Rockafellar:1974}, but with less structure.
Indeed, in \cite[Sect.~4]{Rockafellar:1974} it is supposed that 
the decision set~$\UNCERTAIN$ below is a linear space, in duality with another linear space.
However, for this \S\ref{From_Rockafellians_to_Lagrangians:_bilinear_pairing},
and for the whole paper, we do not require any assumption on the decision
set~$\UNCERTAIN$.

\begin{description}
  \item[Original minimization problem.]
    Suppose given a nonempty set $\UNCERTAIN$, a function
 \( \fonctionuncertain : \UNCERTAIN \to \barRR \) and consider the 
    \emph{original minimization problem}
    \( \inf_{\uncertain \in \UNCERTAIN}\fonctionuncertain\np{\uncertain} \).
 \item[Perturbation scheme, Rockafellian.]
Let be given a vector space~$\PRIMAL$ % , 
   % an element \( \overline\primal \in \PRIMAL \)
   and a function 
\( \Rockafellian : \UNCERTAIN \times \PRIMAL \to \barRR \)
such that
\( \fonctionuncertain\np{\uncertain} =
\Rockafellian\np{\uncertain,0} \),  % \overline\primal}
\( \forall \uncertain \in \UNCERTAIN \).
The variables in~$\PRIMAL$ are called perturbations, and
the function~$\Rockafellian$ is called a \emph{Rockafellian} \cite[Chapter~5]{Royset-Wets:2021}
(a \emph{dualizing parameterization} in
\cite[Definition~11.45]{Rockafellar-Wets:1998},
but with the requirement that \( \Rockafellian\np{\uncertain,\primal} \) be convex
lsc in the perturbation variable~$\primal$). %a perturbation scheme, now
\item[Perturbation function.]
  The original minimization problem is now embedded in a family given by the
  \emph{perturbation function} \cite[p.~295]{Rockafellar:1970}
  (called \emph{min-value function} in \cite[p.~264]{Royset-Wets:2021}) \( \varphi : \PRIMAL \to \barRR \) defined by 
  \( \varphi\np{\primal} = \inf_{\uncertain \in \UNCERTAIN}
    \Rockafellian\np{\uncertain,\primal} \), 
  \( \forall  \primal \in \PRIMAL \).
  The original minimization problem corresponds to
  \( \varphi\np{0} =
\inf_{\uncertain \in \UNCERTAIN}
\Rockafellian\np{\uncertain,0}
= 
\inf_{\uncertain \in \UNCERTAIN}\fonctionuncertain\np{\uncertain} \).
\item[Dual problem, dual function, Lagrangian.]
  Suppose that there exists a vector space~$\DUAL$ such that
$\PRIMAL$ and $\DUAL$ are paired by a bilinear form 
\( \proscal{\cdot}{\cdot} : \PRIMAL \times \DUAL \to \RR \).
Then, we obtain a \emph{dual problem} by
\( %\begin{equation}
   \sup_{\dual \in \DUAL} \bp{ -\LFM{\varphi}\np{\dual} } =
   \LFMbi{\varphi}\np{0} \leq \varphi\np{0}
   =\inf_{\uncertain \in \UNCERTAIN}\fonctionuncertain\np{\uncertain} 
   \). %\end{equation}
   In the dual problem appears the so-called \emph{dual function}
   \( \psi =-\LFM{\varphi} : \DUAL \to \barRR \), also given by
    \(
    \psi\np{\dual}=\inf_{\uncertain\in\UNCERTAIN}\Lagrangian\np{\uncertain,\dual} 
    \), \( \forall \dual \in \DUAL \), where the \emph{Lagrangian}
    \( \Lagrangian : \UNCERTAIN\times\DUAL \to \barRR \)
    is defined by\footnote{%
The expression of the Lagrangian is given in
\cite[Equation~(4.2)]{Rockafellar:1974} but with 
$+\FenchelCoupling{\primal}{\dual}$
instead of $-\FenchelCoupling{\primal}{\dual}$.
The expression we adopt is the one in
\cite[Equation~11(16)]{Rockafellar-Wets:1998},
but without requiring that \( \Rockafellian\np{\uncertain,\primal} \) be convex
lsc in the perturbation variable~$\primal$.}
\( \Lagrangian\np{\uncertain,\dual} =
             \inf_{ \primal \in \PRIMAL } \ba{ \Rockafellian\np{\uncertain,\primal} 
               -\FenchelCoupling{\primal}{\dual} } \),
             \( \forall \np{\uncertain,\dual} \in \UNCERTAIN\times\DUAL \).
\end{description}
Thus, %in the case of a bilinear pairing,
in the classical setting of~\cite{Rockafellar:1974},
one can build a Lagrangian from a Rockafellian
as summarized and sketched in Table~\ref{tab:From_Rockafellians_to_Lagrangians:_bilinear_pairing}.
% This is the classical setting of~\cite{Rockafellar:1974}.
We do not provide a proof of the properties stated in
Table~\ref{tab:From_Rockafellians_to_Lagrangians:_bilinear_pairing},
as these properties are well-known~\cite{Rockafellar:1974}
(and as their proof will be a specialization of the proof of
Proposition~\ref{pr:From_Rockafellians_to_Lagrangians:_general_coupling}
in~\S\ref{From_Rockafellians_to_Lagrangians:_general_coupling}
in the special case of the Fenchel bilinear pairing). 

For any bivariate functions \( \Lagrangian : \UNCERTAIN\times\DUAL \to \barRR \)
and \( \Rockafellian : \UNCERTAIN\times\PRIMAL \to \barRR \),
we denote, for all \( \uncertain\in\UNCERTAIN\),
by \( \PartialFunction{\Lagrangian}{\uncertain} : \DUAL \to \barRR \)
and \( \PartialFunction{\Rockafellian}{\uncertain} : \PRIMAL \to \barRR \)
the partial functions obtained by ``freezing''~$\uncertain$.
In \cite[Sect.~29]{Rockafellar:1970}, % p. 291-292
Rockafellar uses the vocable of \emph{bifunction} and the notation
\( \PartialFunction{\Lagrangian}{\uncertain}=\np{\Lagrangian\uncertain} \), 
 \( \PartialFunction{\Rockafellian}{\uncertain}=\np{\Rockafellian\uncertain}
 \). 

\begin{table}[hbtp]
  \centering
  \begin{tabular}{||c||c|c|c|c||}  %{||c||c|c|c||}
    \hline\hline 
sets & optimization 
    & primal 
      & pairing 
    & dual 
    \\
  &  set~$\UNCERTAIN $ & space $\PRIMAL$ &
 \( \PRIMAL \overset{\FenchelCoupling{\cdot}{\cdot}}{\leftrightarrow} \DUAL \)  & space $\DUAL$
    \\ \hline
    variables &   decision  & perturbation &
                                                       \(
                                             \FenchelCoupling{\primal}{\dual} \)  & sensitivity
    \\
 &   \( \uncertain\in\UNCERTAIN \) & \( \primal\in\PRIMAL \) & \(  \in\RR \) &
                                                                                 \(
                                                                                 \dual\in\DUAL \)
     \\ \hline\hline
    bivariate && Rockafellian & & Lagrangian
    \\
    functions  && \( \Rockafellian : \UNCERTAIN\times\PRIMAL \to \barRR \) &
    & \( \Lagrangian : \UNCERTAIN\times\DUAL \to \barRR \)
    \\ \hline
    definition && & & \( \Lagrangian\np{\uncertain,\dual} = \)
    \\
    && & & \( 
\inf_{ \primal \in \PRIMAL } \ba{ \Rockafellian\np{\uncertain,\primal} 
          -\FenchelCoupling{\primal}{\dual} } \)
    \\ \hline
 property  && & &
          \( -\PartialFunction{\Lagrangian}{\uncertain}
          = \SFM{\bp{\PartialFunction{\Rockafellian}{\uncertain}}}{\Fenchelcoupling} \)
    \\ \hline
 property  && & &
          \( -\PartialFunction{\Lagrangian}{\uncertain} \) is \( \Fenchelcoupling'
                  \)-convex 
    \\ 
   && & & (hence \(\PartialFunction{\Lagrangian}{\uncertain} \) is concave usc)
    \\ \hline\hline
    univariate % optimization
    && perturbation function & & dual function
    \\
functions  && \( \varphi : \PRIMAL \to \barRR \) & & \( \psi : \DUAL \to \barRR \)
          \\ \hline
    definition 
    && \(
      \varphi\np{\primal}=\inf_{\uncertain\in\UNCERTAIN}\Rockafellian\np{\uncertain,\primal}
      \) & &
          \(
      \psi\np{\dual}=\inf_{\uncertain\in\UNCERTAIN}\Lagrangian\np{\uncertain,\dual}
             \)
    \\ \hline
 property 
    && & & \( -\psi = \SFM{\varphi}{\Fenchelcoupling} \)
 %    \\ \hline
 % property 
 %    && & & \( \dual \in \subdifferential{}{\varphi}\np{\primal}
 %           \iff \)
 %    \\
 %     && & & \( \varphi\np{\primal} \UppPlus \bp{ -\psi\np{\dual} } =
 %           \FenchelCoupling{\primal}{\dual} \)
             \\
    \hline\hline
  \end{tabular}
  \caption{From Rockafellians to Lagrangians: bilinear pairing
   \label{tab:From_Rockafellians_to_Lagrangians:_bilinear_pairing}}
\end{table}

\subsection{From Lagrangians to Rockafellians: bilinear pairing}
\label{From_Lagrangians_to_Rockafellians:_bilinear_pairing}

Conversely, one can build a Rockafellian from a Lagrangian
as sketched in Table~\ref{tab:From_Lagrangians_to_Rockafellians:_bilinear_pairing}.
This seems to be less classical than the opposite construction recalled in~\S\ref{From_Rockafellians_to_Lagrangians:_bilinear_pairing}, 
although the formula \( \Rockafellian\np{\uncertain,\primal} =
\sup_{  \dual \in \DUAL } \ba{ \Lagrangian\np{\uncertain,\dual} 
          + \FenchelCoupling{\primal}{\dual} } \)
        appears in \cite[Equation~11(17)]{Rockafellar-Wets:1998},
        as a consequence of \cite[Equation~11(16)]{Rockafellar-Wets:1998}
        and of the requirement that \( \Rockafellian\np{\uncertain,\primal} \) be convex
        lsc in the perturbation variable~$\primal$.
        Note also that \cite[Chapter~33]{Rockafellar:1970} establishes a relationship between the class
of concave-convex saddle functions (Lagrangians of convex programs being a
subclass) and the class of convex Rockafellians via such formula
(see especially \cite[Corollary~33.12]{Rockafellar:1970}).

The properties stated in
Table~\ref{tab:From_Lagrangians_to_Rockafellians:_bilinear_pairing}
will be a consequence of the results of
Proposition~\ref{pr:From_Lagrangians_to_Rockafellians:_general_coupling}
in~\S\ref{From_Lagrangians_to_Rockafellians:_general_coupling},
in the special case of the Fenchel bilinear pairing. 
\begin{table}[hbtp]
  \centering
  \begin{tabular}{||c||c|c|c|c||}  %{||c||c|c|c||}
    \hline\hline 
sets & optimization 
    & primal 
      & pairing 
    & dual 
    \\
  &  set~$\UNCERTAIN $ & space $\PRIMAL$ &
 \( \PRIMAL \overset{\FenchelCoupling{\cdot}{\cdot}}{\leftrightarrow} \DUAL \)  & space $\DUAL$
    \\ \hline
    variables &   decision  & perturbation &
                                                       \(
                                             \FenchelCoupling{\primal}{\dual} \)  & sensitivity
    \\
 &   \( \uncertain\in\UNCERTAIN \) & \( \primal\in\PRIMAL \) & \(  \in\RR \) &
                                                                                 \(
                                                                                 \dual\in\DUAL \)
     \\ \hline\hline
    bivariate && Rockafellian & & Lagrangian
    \\
    functions  && \( \Rockafellian : \UNCERTAIN\times\PRIMAL \to \barRR \) &
    & \( \Lagrangian : \UNCERTAIN\times\DUAL \to \barRR \)
    \\ \hline
    definition && \( \Rockafellian\np{\uncertain,\primal} = \) & & 
    \\
    && \( 
\sup_{  \dual \in \DUAL } \ba{ \Lagrangian\np{\uncertain,\dual} 
          + \FenchelCoupling{\primal}{\dual} } \) & & 
    \\ \hline
 property  && \( \PartialFunction{\Rockafellian}{\uncertain}
          =
              \SFMr{\bp{-\PartialFunction{\Lagrangian}{\uncertain}}}{\Fenchelcoupling} \) & &
    \\ \hline
 property  && \( \PartialFunction{\Rockafellian}{\uncertain} \) is \( \Fenchelcoupling
                  \)-convex & &
    \\ 
   && (hence \(\PartialFunction{\Rockafellian}{\uncertain} \) is convex lsc) & & 
    \\ \hline\hline
    univariate % optimization
    && perturbation function & & dual function
    \\
functions  && \( \varphi : \PRIMAL \to \barRR \) & & \( \psi : \DUAL \to \barRR \)
          \\ \hline
    definition 
    && \(
      \varphi\np{\primal}=\inf_{\uncertain\in\UNCERTAIN}\Rockafellian\np{\uncertain,\primal}
      \) & &
          \(
      \psi\np{\dual}=\inf_{\uncertain\in\UNCERTAIN}\Lagrangian\np{\uncertain,\dual}
             \)
    \\ \hline
 property 
    && \( \varphi \geq \SFMr{\np{-\psi}}{\Fenchelcoupling} \) & & 
 %    \\ \hline
 % property 
 %    && \( \primal \in \Midsubdifferential{}{\np{-\psi}}\np{\dual}
 %           \iff \) & & 
 %    \\
 %     && \( \varphi\np{\primal} \UppPlus \bp{ -\psi\np{\dual} } =
 %           \FenchelCoupling{\primal}{\dual} \) & & 
             \\
    \hline\hline
  \end{tabular}
  \caption{From Lagrangians to Rockafellians: bilinear pairing}
   \label{tab:From_Lagrangians_to_Rockafellians:_bilinear_pairing}
\end{table}

\section{Lagrangians and Rockafellians: general coupling case}
\label{Lagrangians_and_Rockafellians:_general_coupling_case}

In~\S\ref{The_general_coupling_case}, we provide background on couplings
and Fenchel-Moreau conjugacies.
In~\S\ref{From_Rockafellians_to_Lagrangians:_general_coupling}, 
we recall and sketch, in one table, how one can build a Lagrangian from a Rockafellian  \cite{Balder:1977}.
In~\S\ref{From_Lagrangians_to_Rockafellians:_general_coupling}, 
we sketch, in one table, how one can build a Rockafellian from a Lagrangian.
Finally, in~\S\ref{Duality_between_Lagrangians_and_Rockafellians},
we propose a notion of Lagrangian-Rockafellian couple, and
we formally express duality between Lagrangians and Rockafellians.

\subsection{The general coupling case}
\label{The_general_coupling_case}

When we manipulate functions with values 
in~$\barRR = [-\infty,+\infty] $,
we adopt the Moreau \emph{lower addition} \cite{Moreau:1963a,Moreau:1963b} 
that extends the usual addition with 
\( \np{+\infty} \LowPlus \np{-\infty} = \np{-\infty} \LowPlus \np{+\infty} =
-\infty \), and the Moreau \emph{upper addition} 
that extends the usual addition with 
\( \np{+\infty} \UppPlus \np{-\infty} = \np{-\infty} \UppPlus \np{+\infty} =
+\infty \). 

For background on coupling and conjugacies, we refer the reader to
\cite{Singer:1997,Rubinov:2000,Martinez-Legaz:2005}.
We consider two sets $\PRIMAL$ and $\DUAL$ paired by a \emph{coupling}
\( \coupling : \PRIMAL \times \DUAL \to \barRR \).
Then, one associates \emph{conjugacies} 
from the set \( \barRR^\PRIMAL \) of functions \( \PRIMAL  \to \barRR \)
to the set \( \barRR^\DUAL \)  of functions \( \DUAL  \to \barRR \),
and from \( \barRR^\DUAL \) to \( \barRR^\PRIMAL \) 
as follows.
\begin{subequations}
    The \emph{$\coupling$-Fenchel-Moreau conjugate} of a 
    function \( \fonctionprimal : \PRIMAL  \to \barRR \)
    is the function \( \SFM{\fonctionprimal}{\coupling} : \DUAL  \to \barRR \) 
    defined by
    \begin{equation}
      \SFM{\fonctionprimal}{\coupling}\np{\dual} = 
      \sup_{\primal \in \PRIMAL} \Bp{ \coupling\np{\primal,\dual} 
        \LowPlus \bp{ -\fonctionprimal\np{\primal} } } 
      \eqsepv \forall \dual \in \DUAL
      \eqfinp
      \label{eq:Fenchel-Moreau_conjugate}
    \end{equation}
    With the coupling $\coupling$, we associate 
    the \emph{reverse coupling~$\coupling'$} defined by 
    \begin{equation}
      \coupling': \DUAL \times \PRIMAL \to \barRR 
      \eqsepv
      \coupling'\np{\dual,\primal}= \coupling\np{\primal,\dual} 
      \eqsepv
      \forall \np{\dual,\primal} \in \DUAL \times \PRIMAL
      \eqfinp
      \label{eq:reverse_coupling}
    \end{equation}
    The \emph{$\coupling'$-Fenchel-Moreau conjugate} of a 
    function \( \fonctiondual : \DUAL \to \barRR \) is
    the function \( \SFM{\fonctiondual}{\coupling'} : \PRIMAL \to \barRR \) 
    defined by
    \begin{equation}
      \SFM{\fonctiondual}{\coupling'}\np{\primal} = 
      \sup_{ \dual \in \DUAL } \Bp{ \coupling'\np{\dual,\primal} 
        \LowPlus \bp{ -\fonctiondual\np{\dual} } } = 
      \sup_{ \dual \in \DUAL } \Bp{ \coupling\np{\primal,\dual} 
        \LowPlus \bp{ -\fonctiondual\np{\dual} } } 
      \eqsepv \forall \primal \in \PRIMAL 
      \eqfinp
      \label{eq:Fenchel-Moreau_reverse_conjugate}
    \end{equation}
    The \emph{$\coupling$-Fenchel-Moreau biconjugate} of a 
    function \( \fonctionprimal : \PRIMAL  \to \barRR \) is
    the function \( \SFMbi{\fonctionprimal}{\coupling} : \PRIMAL \to \barRR \) 
    defined by
    \begin{equation}
      \SFMbi{\fonctionprimal}{\coupling}\np{\primal} = 
      \bp{\SFM{\fonctionprimal}{\coupling}}^{\coupling'} \np{\primal} = 
      \sup_{ \dual \in \DUAL } \Bp{ \coupling\np{\primal,\dual} 
        \LowPlus \bp{ -\SFM{\fonctionprimal}{\coupling}\np{\dual} } } 
      \eqsepv \forall \primal \in \PRIMAL 
      \eqfinp
      \label{eq:Fenchel-Moreau_biconjugate}
    \end{equation}
   The \emph{$\coupling'$-Fenchel-Moreau biconjugate} of a 
    function \( \fonctiondual : \DUAL \to \barRR \) is
    the function \( \SFMrbi{\fonctiondual}{\coupling} : \DUAL \to \barRR \)
    defined by
    \begin{equation}
      \SFMrbi{\fonctiondual}{\coupling} \np{\dual} =
      \bp{\SFM{\fonctiondual}{\coupling'}}^{\coupling} \np{\dual} =
     \sup_{\primal \in \PRIMAL} \Bp{ \coupling\np{\primal,\dual} 
        \LowPlus \bp{ - \SFM{\fonctiondual}{\coupling'}\np{\primal} }}
      \eqsepv \forall \dual \in \DUAL
      \eqfinp
      \label{eq:Fenchel-Moreau_reverse_biconjugate}
    \end{equation}
         \label{eq:Fenchel-Moreau_conjugate_all}
  \end{subequations}
  A function \( \fonctionprimal : \PRIMAL  \to \barRR \)
is said to be
\( \coupling \)-convex if
\( \SFMbi{\fonctionprimal}{\coupling} =\fonctionprimal \)
(which is equivalent to \( \fonctionprimal =\SFM{\fonctiondual}{\coupling'} \)
for some \( \fonctiondual : \DUAL \to \barRR \)).
A function
\( \fonctiondual : \DUAL \to \barRR \) is said to be
\( \coupling' \)-convex if
\( \SFMrbi{\fonctiondual}{\coupling} =\fonctiondual \)
(which is equivalent to \( \fonctiondual = \SFM{\fonctionprimal}{\coupling} \)
for some \( \fonctionprimal : \PRIMAL  \to \barRR \)).
%
% We define the \emph{\(\coupling\)-subdifferential}
% \( \Midsubdifferential{\coupling}{\fonctionprimal}\np{\primal} \)
% of a function \( \fonctionprimal : \PRIMAL  \to \barRR \)
% at~\( \primal \in \PRIMAL \) by 
% \( \dual \in \Midsubdifferential{\coupling}{\fonctionprimal}\np{\primal} 
%            \iff \fonctionprimal\np{\primal} \UppPlus \SFM{\fonctionprimal}{\coupling}\np{\dual} =
%            \Coupling{\primal}{\dual} \).
% We define the \emph{\(\coupling'\)-subdifferential}
% \( \Midsubdifferential{\coupling'}{\fonctiondual}\np{\dual} \)
% of a function \( \fonctiondual : \DUAL  \to \barRR \)
% at~\( \dual \in \DUAL \) by 
% \( \primal \in \Midsubdifferential{\coupling'}{\fonctiondual}\np{\dual} 
%            \iff \fonctiondual\np{\dual} \UppPlus \SFMr{\fonctiondual}{\coupling}\np{\primal} =
%            \Coupling{\primal}{\dual} \).

\subsection{From Rockafellians to Lagrangians: general coupling}
\label{From_Rockafellians_to_Lagrangians:_general_coupling}

Following \cite{Balder:1977} --- and taking inspiration from
\S\ref{From_Rockafellians_to_Lagrangians:_bilinear_pairing}
(see also \cite[Sect.~3]{Martinez-Legaz:2005}) --- 
one can build a Lagrangian from a Rockafellian
as sketched in
Table~\ref{tab:From_Rockafellians_to_Lagrangians:_general_coupling}.
The formal statement is given in
Proposition~\ref{pr:From_Rockafellians_to_Lagrangians:_general_coupling}.
The result is not new \cite[Sect.~3]{Martinez-Legaz:2005}, but we give a proof
for the sake of completeness and symmetry in the exposition. 

\begin{table}[hbtp]
  \centering
  \begin{tabular}{||c||c|c|c|c||}  %{||c||c|c|c||}
    \hline\hline 
sets & optimization 
    & primal 
      & coupling 
    & dual 
    \\
  &  set~$\UNCERTAIN $ & set $\PRIMAL$ &
 \( \PRIMAL \overset{\coupling}{\leftrightarrow} \DUAL \)  & set~$\DUAL$
    \\ \hline
    variables &   decision  & perturbation &
                                                       \(
                                             \Coupling{\primal}{\dual} \)  & sensitivity
    \\
 &   \( \uncertain\in\UNCERTAIN \) & \( \primal\in\PRIMAL \) & \(  \in\barRR \) &
                                                                                 \(
                                                                                 \dual\in\DUAL \)
 %    variables &   decision~$\uncertain$  & perturbation~$\primal$ &
 % \( \Coupling{\primal}{\dual} \in\barRR\)  & sensitivity~$\dual$
     \\ \hline\hline
    bivariate && Rockafellian & & Lagrangian
    \\
    functions  && \( \Rockafellian : \UNCERTAIN\times\PRIMAL \to \barRR \) &
    & \( \Lagrangian : \UNCERTAIN\times\DUAL \to \barRR \)
    \\ \hline
    definition && & & \( \Lagrangian\np{\uncertain,\dual} = \)
    \\
    && & & \( 
\inf_{ \primal \in \PRIMAL } \Ba{ \Rockafellian\np{\uncertain,\primal} 
          \UppPlus \bp{ -\Coupling{\primal}{\dual} } } \)
    \\ \hline
 property  && & &
          \( -\PartialFunction{\Lagrangian}{\uncertain}
          = \SFM{\bp{\PartialFunction{\Rockafellian}{\uncertain}}}{\coupling} \)
    \\ \hline
 property  && & &
          \( -\PartialFunction{\Lagrangian}{\uncertain} \) is \( \coupling'
                  \)-convex 
    \\ \hline\hline
    univariate % optimization
    && perturbation function & & dual function
    \\
functions  && \( \varphi : \PRIMAL \to \barRR \) & & \( \psi : \DUAL \to \barRR \)
          \\ \hline
    definition 
    && \(
      \varphi\np{\primal}=\inf_{\uncertain\in\UNCERTAIN}\Rockafellian\np{\uncertain,\primal}
      \) & &
          \(
      \psi\np{\dual}=\inf_{\uncertain\in\UNCERTAIN}\Lagrangian\np{\uncertain,\dual}
             \)
    \\ \hline
 property 
    && & & \( -\psi = \SFM{\varphi}{\coupling} \)
 %    \\ \hline
 % property 
 %    && & & \( \dual \in \Midsubdifferential{\coupling}{\varphi}\np{\primal}
 %           \iff \)
 %    \\
 %     && & & \( \varphi\np{\primal} \UppPlus \bp{ -\psi\np{\dual} } =
 %           \Coupling{\primal}{\dual} \)
              \\
    \hline\hline
  \end{tabular}
  \caption{From Rockafellians to Lagrangians: general coupling}
   \label{tab:From_Rockafellians_to_Lagrangians:_general_coupling}
\end{table}

\begin{proposition}[from \cite{Balder:1977,Martinez-Legaz:2005}]
  \label{pr:From_Rockafellians_to_Lagrangians:_general_coupling}
We consider a set~$\UNCERTAIN$, and two sets $\PRIMAL$ and $\DUAL$ paired by a coupling 
\( \coupling : \PRIMAL \times \DUAL \to \barRR \).
We consider a Rockafellian
\( \Rockafellian : \UNCERTAIN\times\PRIMAL \to\barRR\) 
and we define the \emph{Lagrangian}
\( \Lagrangian : \UNCERTAIN\times\DUAL \to \barRR \)
by\footnote{%
Our definition is close to \cite[p.~252]{Martinez-Legaz:2005}.}
\begin{equation}
  \Lagrangian\np{\uncertain,\dual} = 
\inf_{ \primal \in \PRIMAL } \Ba{ \Rockafellian\np{\uncertain,\primal} 
  \UppPlus \bp{ -\Coupling{\primal}{\dual} } }
\eqsepv \forall \np{\uncertain,\dual} \in \UNCERTAIN\times\DUAL
\eqfinv
\label{eq:From_Rockafellians_to_Lagrangians:_general_coupling}
\end{equation}
the \emph{perturbation function} \( \varphi : \PRIMAL \to
                  \barRR \) and the \emph{dual function}                 
\( \psi : \DUAL \to \barRR \) by 
\begin{equation}
  \varphi\np{\primal} =
  \inf_{\uncertain \in \UNCERTAIN}\Rockafellian\np{\uncertain,\primal}
  \eqsepv  \forall  \primal \in \PRIMAL
  \mtext{ and }
  \psi\np{\dual}=\inf_{\uncertain\in\UNCERTAIN}\Lagrangian\np{\uncertain,\dual} 
  \eqsepv  \forall \dual \in \DUAL
  \eqfinp
  \label{eq:value_function_dual_function}
\end{equation}
Then, we have that
 \( -\PartialFunction{\Lagrangian}{\uncertain}
          = \SFM{\bp{\PartialFunction{\Rockafellian}{\uncertain}}}{\coupling} \) and
          \( -\PartialFunction{\Lagrangian}{\uncertain} \) is \( \coupling'
          \)-convex, for all \( \uncertain\in\UNCERTAIN\), 
          and that \( -\psi = \SFM{\varphi}{\coupling} \).
        \end{proposition}

        \begin{proof}
The equality \( -\PartialFunction{\Lagrangian}{\uncertain}
          = \SFM{\bp{\PartialFunction{\Rockafellian}{\uncertain}}}{\coupling}\),
 for all \( \uncertain\in\UNCERTAIN\), 
is a straightforward application of
definition~\eqref{eq:Fenchel-Moreau_conjugate}
with \( \fonctionprimal = \PartialFunction{\Rockafellian}{\uncertain} \)
and of
\( -\Lagrangian\np{\uncertain,\dual} = 
\sup_{ \primal \in \PRIMAL } \Ba{ \Coupling{\primal}{\dual} 
  \LowPlus \bp{ -\Rockafellian\np{\uncertain,\primal} } } \),
for all \( \np{\uncertain,\dual} \in \UNCERTAIN\times\DUAL \). 

By property of conjugacies \cite[Proposition~6.1 (ii)]{Martinez-Legaz:2005},
we have that
\( \SFM{\bp{-\PartialFunction{\Lagrangian}{\uncertain}}}{\coupling'\coupling}
= \SFM{\bp{\PartialFunction{\Rockafellian}{\uncertain}}}{\coupling\coupling'\coupling}
=
\SFM{\bp{\PartialFunction{\Rockafellian}{\uncertain}}}{\coupling}
= -\PartialFunction{\Lagrangian}{\uncertain}
\),
for all \( \uncertain\in\UNCERTAIN\).
Hence, \( -\PartialFunction{\Lagrangian}{\uncertain} \) is \( \coupling'
\)-convex, for all \( \uncertain\in\UNCERTAIN\).

Finally, as conjugacies turn an infimum into a supremum, we have that
\( \SFM{\varphi}{\coupling} =
 \SFM{ \bp{ \inf_{\uncertain\in\UNCERTAIN}\Rockafellian\np{\uncertain,\cdot}}}{\coupling} =
 \sup_{\uncertain\in\UNCERTAIN} \SFM{ \bp{\Rockafellian\np{\uncertain,\cdot}}}{\coupling} =
 \sup_{\uncertain\in\UNCERTAIN} \bp{-\PartialFunction{\Lagrangian}{\uncertain}}=
 -  \inf_{\uncertain\in\UNCERTAIN} \PartialFunction{\Lagrangian}{\uncertain}=
 -\psi \) by~\eqref{eq:value_function_dual_function}.
         \end{proof}

         The dual problem is \( \SFMbi{\varphi}{\coupling}\np{\primal} =
         \SFMr{\np{-\psi}}{\coupling}\np{\primal} =
    \sup_{ \dual \in \DUAL } \bp{ \coupling\np{\primal,\dual} 
      \LowPlus \psi\np{\dual} } \).
    In the Fenchel bilinear pairing case, and in \cite{Rockafellar:1974},
    we have that \( \primal=0 \) and \( \coupling\np{0,\dual} =
    \FenchelCoupling{0}{\dual} =0\), \( \forall \dual \in \DUAL \).
    This is generally no longer the case in the general coupling case.

\subsection{From Lagrangians to Rockafellians: general coupling}
\label{From_Lagrangians_to_Rockafellians:_general_coupling}

Taking inspiration from
\S\ref{From_Lagrangians_to_Rockafellians:_bilinear_pairing}, 
we show one can build a Rockafellian from a Lagrangian
as sketched in
Table~\ref{tab:From_Lagrangians_to_Rockafellians:_general_coupling}.
The formal statement is given in
Proposition~\ref{pr:From_Lagrangians_to_Rockafellians:_general_coupling},
wich seems to be new. 

\begin{table}[hbtp]
  \centering
  \begin{tabular}{||c||c|c|c|c||}  %{||c||c|c|c||}
    \hline\hline 
sets & optimization 
    & primal 
      & coupling 
    & dual 
    \\
  &  set~$\UNCERTAIN $ & set $\PRIMAL$ &
 \( \PRIMAL \overset{\coupling}{\leftrightarrow} \DUAL \)  & set~$\DUAL$
    \\ \hline
    variables &   decision  & perturbation &
                                                       \(
                                             \Coupling{\primal}{\dual} \)  & sensitivity
    \\
 &   \( \uncertain\in\UNCERTAIN \) & \( \primal\in\PRIMAL \) & \(  \in\barRR \) &
                                                                                 \(
                                                                                 \dual\in\DUAL \)
 %    variables &   decision~$\uncertain$  & perturbation~$\primal$ &
 % \( \Coupling{\primal}{\dual} \in\barRR\)  & sensitivity~$\dual$
     \\ \hline\hline
    bivariate && Rockafellian & & Lagrangian
    \\
    functions  && \( \Rockafellian : \UNCERTAIN\times\PRIMAL \to \barRR \) &
    & \( \Lagrangian : \UNCERTAIN\times\DUAL \to \barRR \)
    \\ \hline
    definition && \( \Rockafellian\np{\uncertain,\primal} = \) & & 
    \\
    && \( 
\sup_{  \dual \in \DUAL } \ba{ \Lagrangian\np{\uncertain,\dual} 
          \LowPlus \Coupling{\primal}{\dual} } \) & & 
    \\ \hline
 property  && \( \PartialFunction{\Rockafellian}{\uncertain}
          = \SFM{\bp{-\PartialFunction{\Lagrangian}{\uncertain}}}{\coupling'} \)
    & &
          
    \\ \hline
 property  && 
          \( \PartialFunction{\Rockafellian}{\uncertain} \) is \( \coupling
                  \)-convex & &
    \\ \hline\hline
    univariate % optimization
    && perturbation function & & dual function
    \\
functions  && \( \varphi : \PRIMAL \to \barRR \) & & \( \psi : \DUAL \to \barRR \)
          \\ \hline
    definition 
    && \(
      \varphi\np{\primal}=\inf_{\uncertain\in\UNCERTAIN}\Rockafellian\np{\uncertain,\primal}
      \) & &
          \(
      \psi\np{\dual}=\inf_{\uncertain\in\UNCERTAIN}\Lagrangian\np{\uncertain,\dual}
             \)
    \\ \hline
 property 
    && \( \varphi \geq \SFMr{\np{-\psi}}{\coupling} \) & & 
 %    \\ \hline
 % property 
 %    && \( \primal \in \Midsubdifferential{\coupling'}{\np{-\psi}}\np{\dual}
 %           \iff \) & & 
 %    \\
 %     && \( \varphi\np{\primal} \UppPlus \bp{ -\psi\np{\dual} } =
 %           \Coupling{\primal}{\dual} \) & & 
             \\
    \hline\hline
  \end{tabular}
  \caption{From Lagrangians to Rockafellians: general coupling}
  \label{tab:From_Lagrangians_to_Rockafellians:_general_coupling}
\end{table}

Contrarily to
Proposition~\ref{pr:From_Rockafellians_to_Lagrangians:_general_coupling} ---
where we obtained the equality \( -\psi = \SFM{\varphi}{\coupling} \)
between dual and perturbation functions --- 
we now only obtain the inequality \( \varphi \geq \SFMr{-\psi}{\coupling} \).
This is because the definition~\eqref{eq:From_Lagrangians_to_Rockafellians:_general_coupling}
of a Rockafellian built from a Lagrangian
involves a supremum operation, which does not behave, with a conjugacy,
as nicely as an infimum operation does. 

\begin{proposition}
    \label{pr:From_Lagrangians_to_Rockafellians:_general_coupling}
We consider a set~$\UNCERTAIN$, and two sets $\PRIMAL$ and $\DUAL$ paired by a coupling 
\( \coupling : \PRIMAL \times \DUAL \to \barRR \).
We consider a Lagrangian
\( \Lagrangian : \UNCERTAIN\times\DUAL \to \barRR \)
and we define the \emph{Rockafellian}
\( \Rockafellian : \UNCERTAIN\times\PRIMAL \to\barRR\) 
by
\begin{equation}
\Rockafellian\np{\uncertain,\primal} = 
\sup_{  \dual \in \DUAL } \ba{ \Lagrangian\np{\uncertain,\dual} 
  \LowPlus \Coupling{\primal}{\dual} }
\eqsepv \forall \np{\uncertain,\primal} \in \UNCERTAIN\times\PRIMAL
\eqfinv
\label{eq:From_Lagrangians_to_Rockafellians:_general_coupling}
\end{equation}
and the \emph{perturbation function} \( \varphi : \PRIMAL \to
                  \barRR \) and the \emph{dual function}                 
                  \( \psi : \DUAL \to \barRR \) as in~\eqref{eq:value_function_dual_function}.

                  Then, we have that
\( \PartialFunction{\Rockafellian}{\uncertain}
          = \SFM{\bp{-\PartialFunction{\Lagrangian}{\uncertain}}}{\coupling'} \)
and
          \( \PartialFunction{\Rockafellian}{\uncertain} \) is \( \coupling
                  \)-convex, for all \( \uncertain\in\UNCERTAIN\), 
          and that \( \varphi \geq \SFMr{-\psi}{\coupling} \).
        \end{proposition}

        \begin{proof}
The equality \( \PartialFunction{\Rockafellian}{\uncertain}
          = \SFM{\bp{-\PartialFunction{\Lagrangian}{\uncertain}}}{\coupling'} \)
 for all \( \uncertain\in\UNCERTAIN\), 
is a straightforward application of
definition~\eqref{eq:Fenchel-Moreau_reverse_conjugate}
with \( \fonctiondual=-\PartialFunction{\Lagrangian}{\uncertain} \)
and \( \PartialFunction{\Rockafellian}{\uncertain} \)
given by~\eqref{eq:From_Lagrangians_to_Rockafellians:_general_coupling}.

By property of conjugacies \cite[Proposition~6.1 (ii)]{Martinez-Legaz:2005}, we have that
\( \SFM{\bp{\PartialFunction{\Rockafellian}{\uncertain}}}{\coupling\coupling'}
= \SFM{\bp{-\PartialFunction{\Lagrangian}{\uncertain}}}{\coupling'\coupling\coupling'}
=
\SFM{\bp{-\PartialFunction{\Lagrangian}{\uncertain}}}{\coupling'}
= \PartialFunction{\Rockafellian}{\uncertain}
\),
for all \( \uncertain\in\UNCERTAIN\).
Hence, \( \PartialFunction{\Rockafellian}{\uncertain} \) is \( \coupling
\)-convex, for all \( \uncertain\in\UNCERTAIN\).

Finally, as conjugacies turn an infimum into a supremum,
for all \( \uncertain\in\UNCERTAIN\) we have that
\( \SFM{\varphi}{\coupling} =
 \SFM{ \bp{ \inf_{\uncertain\in\UNCERTAIN}\Rockafellian\np{\uncertain,\cdot}}}{\coupling} =
 \sup_{\uncertain\in\UNCERTAIN} \SFM{ \bp{\Rockafellian\np{\uncertain,\cdot}}}{\coupling} =
 \sup_{\uncertain\in\UNCERTAIN} \SFM{ \bp{-\PartialFunction{\Lagrangian}{\uncertain}}}{\coupling\coupling'}
 \leq
 \sup_{\uncertain\in\UNCERTAIN} \bp{-\PartialFunction{\Lagrangian}{\uncertain}}
 =-\inf_{\uncertain\in\UNCERTAIN} \PartialFunction{\Lagrangian}{\uncertain}=
 -\psi \) by~\eqref{eq:value_function_dual_function},
 and where we have used that \( \SFM{
   \bp{-\PartialFunction{\Lagrangian}{\uncertain}}}{\coupling\coupling'}
 \leq -\PartialFunction{\Lagrangian}{\uncertain} \).
 Then, we get that \( \SFMbi{\varphi}{\coupling} \geq
 \SFMr{\np{-\psi}}{\coupling} \),
 hence that \( \varphi \geq \SFMbi{\varphi}{\coupling} \geq
 \SFMr{\np{-\psi}}{\coupling} \).
         \end{proof}

         \subsection{Duality between Lagrangians and Rockafellians}
\label{Duality_between_Lagrangians_and_Rockafellians}

         To formally express duality between Lagrangians and Rockafellians,
         we take inspiration from \cite[Sect.~8]{Moreau:1966-1967}, especially
         the notions of \emph{fonctions sur-duales} and of
         \emph{fonctions duales} (dual functions), that is,
         minimal elements in a generalized Fenchel-Young type inequality
         (see also \cite[p.~104-105]{Rockafellar:1970}). 
         The following definition and theorem are new. 

         \begin{definition}
We consider a set~$\UNCERTAIN$, and two sets $\PRIMAL$ and $\DUAL$ paired by a coupling 
\( \coupling : \PRIMAL \times \DUAL \to \barRR \).
We equip \( \barRR^{\UNCERTAIN\times\PRIMAL} \times \barRR^{\UNCERTAIN\times\DUAL} \) with the natural ordering.
We say that the functions
\( \Lagrangian : \UNCERTAIN\times\DUAL \to \barRR \)
and \( \Rockafellian : \UNCERTAIN\times\PRIMAL \to \barRR \)
form a \emph{Lagrangian-Rockafellian} couple
\( \np{\Lagrangian,\Rockafellian} \) if
\( \np{-\Lagrangian,\Rockafellian} \)
is minimal in the inequality
\begin{equation}
\bp{ -\Lagrangian\np{\uncertain,\dual} } \UppPlus 
     \Rockafellian\np{\uncertain,\primal} \geq 
     \Coupling{\primal}{\dual}
\eqsepv \forall \np{\uncertain,\primal,\dual} \in \UNCERTAIN\times\PRIMAL\times\DUAL 
   \eqfinp 
\end{equation}
\label{de:Lagrangian-Rockafellian_couple}
         \end{definition}

         We can now state our main result, which establishes a duality between
         Lagrangians and Rockafellians,
         expressed in different equivalent forms:
         as minimal elements in a generalized Fenchel-Young type inequality;
by means of Fenchel-Moreau conjugates dual pairs; and also by means of 
generalized convex functions (equal to their Fenchel-Moreau biconjugate).

         \begin{theorem}
We consider a set~$\UNCERTAIN$, two sets $\PRIMAL$ and $\DUAL$ paired by a coupling 
\( \coupling : \PRIMAL \times \DUAL \to \barRR \),
and two functions
\( \Lagrangian : \UNCERTAIN\times\DUAL \to \barRR \)
and \( \Rockafellian : \UNCERTAIN\times\PRIMAL \to \barRR \).
Then, the following are equivalent.
\begin{enumerate}
\item
\label{it:Lagrangian-Rockafellian_couple}
  The functions
\( \Lagrangian \) and \( \Rockafellian \) form a Lagrangian-Rockafellian
couple \( \np{\Lagrangian,\Rockafellian} \).
\item
\label{it:Lagrangian-Rockafellian_infsup}
  Equality~\eqref{eq:From_Rockafellians_to_Lagrangians:_general_coupling}
  and Equality~\eqref{eq:From_Lagrangians_to_Rockafellians:_general_coupling}
  hold true, that is,
  \begin{subequations}
    \begin{align}
      \Lagrangian\np{\uncertain,\dual}
      &= 
\inf_{ \primal \in \PRIMAL } \Ba{ \Rockafellian\np{\uncertain,\primal} 
  \UppPlus \bp{ -\Coupling{\primal}{\dual} } }
\eqsepv \forall \np{\uncertain,\dual} \in \UNCERTAIN\times\DUAL
        \eqfinv
        \\
      \Rockafellian\np{\uncertain,\primal}
      &= 
\sup_{  \dual \in \DUAL } \ba{ \Lagrangian\np{\uncertain,\dual} 
  \LowPlus \Coupling{\primal}{\dual} }
\eqsepv \forall \np{\uncertain,\primal} \in \UNCERTAIN\times\PRIMAL
\eqfinp
    \end{align}
\label{eq:Lagrangian-Rockafellian_infsup}    
  \end{subequations}
\item
  \label{it:Lagrangian-Rockafellian_dual_functions}
  The functions 
  \( -\Lagrangian \) and \( \Rockafellian \) are dual functions,
  with respect to the coupling~$\coupling$, in the sense
that 
  \begin{equation}
-\PartialFunction{\Lagrangian}{\uncertain}
          = \SFM{\bp{\PartialFunction{\Rockafellian}{\uncertain}}}{\coupling} 
          \mtext{ and }
          \PartialFunction{\Rockafellian}{\uncertain}
          = \SFM{\bp{-\PartialFunction{\Lagrangian}{\uncertain}}}{\coupling'} 
          \eqsepv \forall \uncertain\in\UNCERTAIN
          \eqfinp
  \label{eq:Lagrangian-Rockafellian_dual_functions}          
        \end{equation}
      \item
          \label{it:Lagrangian-Rockafellian_Rockafellian_convex}
        We have that
        \begin{equation}
-\PartialFunction{\Lagrangian}{\uncertain}
          = \SFM{\bp{\PartialFunction{\Rockafellian}{\uncertain}}}{\coupling} 
          \mtext{ and }
          \SFMbi{\bp{\PartialFunction{\Rockafellian}{\uncertain}}}{\coupling} 
          =\PartialFunction{\Rockafellian}{\uncertain}
          \eqsepv \forall \uncertain\in\UNCERTAIN
          \eqfinp           
          \label{eq:Lagrangian-Rockafellian_Rockafellian_convex}
        \end{equation}
      \item
                  \label{it:Lagrangian-Rockafellian_Lagrangian_convex}
       We have that
        \begin{equation}
          \PartialFunction{\Rockafellian}{\uncertain}
          = \SFM{\bp{-\PartialFunction{\Lagrangian}{\uncertain}}}{\coupling'} 
          \mtext{ and }
          \SFM{\bp{-\PartialFunction{\Lagrangian}{\uncertain}}}{\coupling'\coupling} 
          =-\PartialFunction{\Lagrangian}{\uncertain}
          \eqsepv \forall \uncertain\in\UNCERTAIN
          \eqfinp           
          \label{eq:Lagrangian-Rockafellian_Lagrangian_convex}
          \end{equation}
        \end{enumerate}
          \label{th:Lagrangian-Rockafellian_Lagrangian}           
         \end{theorem}
        The above equivalences are formal, in that they rather easily follow from definitions
         (Definition~\ref{de:Lagrangian-Rockafellian_couple}, Equation~\eqref{eq:Fenchel-Moreau_conjugate_all}
         in~\S\ref{The_general_coupling_case}) and from basic properties of conjugacies.
         They simply establish, in different equivalent ways, the duality between
         Lagrangians and Rockafellians. In particular, 
         in Item~\ref{it:Lagrangian-Rockafellian_couple}, Item~\ref{it:Lagrangian-Rockafellian_infsup}
         and Item~\ref{it:Lagrangian-Rockafellian_dual_functions},
         Lagrangians and Rockafellians are characterized in a balanced fashion, 
         by contrast with Item~\ref{it:Lagrangian-Rockafellian_Rockafellian_convex}
         and Item~\ref{it:Lagrangian-Rockafellian_Lagrangian_convex}.
       \medskip
       
\begin{proof}

  Item~\ref{it:Lagrangian-Rockafellian_couple}
  $\iff$
  Item~\ref{it:Lagrangian-Rockafellian_infsup}.

  The equivalence follows from Definition~\ref{de:Lagrangian-Rockafellian_couple}
  and from 
  the following equivalences
  \cite{Moreau:1963a,Moreau:1963b} (see also
  \cite[Appendix A.1]{Chancelier-DeLara:2019}): 
  for any \( \np{\uncertain,\primal,\dual} \in
  \UNCERTAIN\times\PRIMAL\times\DUAL \),
  \begin{subequations}
    \begin{align}
& \bp{ -\Lagrangian\np{\uncertain,\dual} } \UppPlus 
     \Rockafellian\np{\uncertain,\primal} \geq 
                    \Coupling{\primal}{\dual}
             \eqfinv
      \\
      \iff &
  \Lagrangian\np{\uncertain,\dual} 
             \leq
             \Rockafellian\np{\uncertain,\primal}             
    \UppPlus
\bp{-\Coupling{\primal}{\dual}}
             \eqfinv
      \\
      \iff &
  \Rockafellian\np{\uncertain,\primal}
  \geq
  \Lagrangian\np{\uncertain,\dual}
  \LowPlus
             \Coupling{\primal}{\dual}
             \eqfinp 
    \end{align}
  \end{subequations}

    Item~\ref{it:Lagrangian-Rockafellian_infsup}
  $\iff$
    Item~\ref{it:Lagrangian-Rockafellian_dual_functions}.

    Equations~\eqref{eq:Lagrangian-Rockafellian_infsup} are equivalent to 
    Equations~\eqref{eq:Lagrangian-Rockafellian_dual_functions}
    by definitions~\eqref{eq:Fenchel-Moreau_conjugate}
    and~\eqref{eq:Fenchel-Moreau_reverse_conjugate}. 
\medskip 

    Item~\ref{it:Lagrangian-Rockafellian_dual_functions}
    $\implies$
    Item~\ref{it:Lagrangian-Rockafellian_Rockafellian_convex}.

    Equations~\eqref{eq:Lagrangian-Rockafellian_dual_functions}
    imply that
\( -\PartialFunction{\Lagrangian}{\uncertain}
= \SFM{\bp{\PartialFunction{\Rockafellian}{\uncertain}}}{\coupling} \)
and 
    \( \SFM{\bp{\PartialFunction{\Rockafellian}{\uncertain}}}{\coupling\coupling'}
    =\SFM{\Bp{\SFM{\bp{\PartialFunction{\Rockafellian}{\uncertain}}}{\coupling}}}{\coupling'}
= \SFM{\bp{-\PartialFunction{\Lagrangian}{\uncertain}}}{\coupling'}
= \PartialFunction{\Rockafellian}{\uncertain}
\),
for all \( \uncertain\in\UNCERTAIN\).
Hence, we obtain
Equations~\eqref{eq:Lagrangian-Rockafellian_Rockafellian_convex}.
\medskip 

Item~\ref{it:Lagrangian-Rockafellian_Rockafellian_convex}
$\implies$
Item~\ref{it:Lagrangian-Rockafellian_dual_functions}.

Equations~\eqref{eq:Lagrangian-Rockafellian_Rockafellian_convex}
imply that
\( -\PartialFunction{\Lagrangian}{\uncertain}
= \SFM{\bp{\PartialFunction{\Rockafellian}{\uncertain}}}{\coupling} \)
and 
\( \PartialFunction{\Rockafellian}{\uncertain}
= \SFM{\bp{\PartialFunction{\Rockafellian}{\uncertain}}}{\coupling\coupling'}
=\SFM{\Bp{\SFM{\bp{\PartialFunction{\Rockafellian}{\uncertain}}}{\coupling}}}{\coupling'}
= \SFM{\bp{-\PartialFunction{\Lagrangian}{\uncertain}}}{\coupling'} \),
for all \( \uncertain\in\UNCERTAIN\).
Hence, we obtain
Equations~\eqref{eq:Lagrangian-Rockafellian_dual_functions}. 

\medskip 

    Item~\ref{it:Lagrangian-Rockafellian_dual_functions}
    $\iff$
    Item~\ref{it:Lagrangian-Rockafellian_Lagrangian_convex}.
    
    This equivalence is shown by two implications
    \eqref{eq:Lagrangian-Rockafellian_dual_functions}
    $\implies$
    \eqref{eq:Lagrangian-Rockafellian_Lagrangian_convex}
    and
     \eqref{eq:Lagrangian-Rockafellian_Lagrangian_convex}
     $\implies$
     \eqref{eq:Lagrangian-Rockafellian_dual_functions}
    as above. 
  
\end{proof}

As a consequence, for a Lagrangian-Rockafellian couple
\( \np{\Lagrangian,\Rockafellian} \), one necessarily has that 
\( -\PartialFunction{\Lagrangian}{\uncertain} \) is \( \coupling'\)-convex
and
\( \PartialFunction{\Rockafellian}{\uncertain} \) is \( \coupling\)-convex,
for all \( \uncertain\in\UNCERTAIN\).

\section{Conclusion}

In this paper, we have highlighted a duality between Lagrangians and
Rockafellians, as these two functions appear in the
``perturbation + duality'' method of \cite{Rockafellar:1974}.
We have treated both functions equally, and have provided
formulas that go both way: from Rockafellian to Lagrangian (classical);
from Lagrangian to Rockafellian (less classical).
The setting is the one of so-called abstract or generalized convexity,
  where the perturbation belongs to a primal set paired, by a coupling
  function, with a dual set. It encompasses the classical Fenchel duality
  setting.

  We have proposed a notion of Lagrangian-Rockafellian couple ---
     minimal elements in a generalized Fenchel-Young type inequality ---
and
we have formally expressed duality between Lagrangians and Rockafellians by
means of Fenchel-Moreau conjugates dual pairs, and also by means of 
generalized convex functions (equal to their Fenchel-Moreau biconjugate).
\bigskip

\textbf{Acknowledgements.}
The Author thanks Johannes Royset and Roger Wets for their comments on a first version of
the paper.
The Author thanks the Reviewer for her/his comments that helped improve the paper,
notably by clarifying the contribution and by pointing suitable sources.

% \bibliographystyle{abbrv}
% \bibliography{./DeLara}

\newcommand{\noopsort}[1]{} \ifx\undefined\allcaps\def\allcaps#1{#1}\fi

\end{document}